\documentclass[reqno,a4paper]{amsart}
\usepackage{amssymb}
\usepackage{amsmath}
\begin{document} 
\newtheorem{prop}{Proposition}[section]
\newtheorem{Def}{Definition}[section] \newtheorem{theorem}{Theorem}[section]
\newtheorem{lemma}{Lemma}[section] \newtheorem{Cor}{Corollary}[section]

\title[Chern-Simons-Higgs in temporal gauge]{\bf Global well-posedness in energy space for the Chern-Simons-Higgs system in temporal gauge}
\author[Hartmut Pecher]{
{\bf Hartmut Pecher}\\
Fachbereich Mathematik und Naturwissenschaften\\
Bergische Universit\"at Wuppertal\\
Gau{\ss}str.  20\\
42119 Wuppertal\\
Germany\\
e-mail {\tt pecher@math.uni-wuppertal.de}}
\date{}

\begin{abstract}
The Cauchy problem for the Chern-Simons-Higgs system in the (2+1)-dimensional Minkowski space in temporal gauge is globally well-posed in energy space improving a result of Huh. The proof uses the bilinear space-time estimates in wave-Sobolev spaces by d'Ancona, Foschi and Selberg, an $L^6_x L^2_t$-estimate for solutions of the wave equation, and also takes advantage of a null condition.
\end{abstract}
\maketitle
\renewcommand{\thefootnote}{\fnsymbol{footnote}}
\footnotetext{\hspace{-1.5em}{\it 2000 Mathematics Subject Classification:} 
35Q40, 35L70 }
\normalsize 
\setcounter{section}{0}
\section{Introduction and main results}
\noindent Consider the Chern-Simons-Higgs system in the Minkowski space 
${\mathbb R}^{1+2} = {\mathbb R}_t \times {\mathbb R}_x^2$ with metric $g_{\mu \nu} = diag(1,-1,-1)$ :
\begin{align}
\label{*1}
 F_{\mu \nu} & =  2 \epsilon^{\mu \nu \rho} Im(\overline{\phi} D^{\rho} \phi) \\
\label{*2}
D_{\mu} D^{\mu} \phi & = - \phi V'(|\phi|^2) \, ,
\end{align}
with initial data
\begin{equation}
\label{*3}
A_{\nu}(0) = a_{\nu} \, , \, \phi(0) = \phi_0 \, , \, (\partial_t \phi)(0) = \phi_1 \, , 
\end{equation}
where we use the convention that repeated upper and lower indices are summed, Greek indices run over 0,1,2 and Latin indices over 1,2. Here 
\begin{align*}
D^{\mu}  & := \partial_{\mu} - iA_{\mu} \\
 F_{\mu \nu} & := \partial_{\mu} A_{\nu} - \partial_{\nu} A_{\mu} \\
\end{align*}
$F_{\mu \nu} : {\mathbb R}^{1+2} \to {\mathbb R}$ denotes the curvature, $\phi : {\mathbb R}^{1+2} \to {\mathbb C}$ is a scalar field and $A_{\nu} : {\mathbb R}^{1+2} \to {\mathbb R}$ are the gauge potentials. We use the notation $\partial_{\mu} = \frac{\partial}{\partial x_{\mu}}$, where we write $(x^0,x^1,...,x^n) = (t,x^1,...,x^n)$ and also $\partial_0 = \partial_t$ and $\nabla = (\partial_1,\partial_2)$. $\epsilon^{\mu \nu \rho}$ is the totally skew-symmetric tensor with $\epsilon^{012} = 1$, and the Higgs potential $V$ is assumed to fulfill
$V \in C^{\infty}({\mathbb R}^+,{\mathbb R}) $ , $V(0)=0$ and all derivatives of $V$ have polynomial growth.

The energy $E(t)$ of the system is conserved, where
$$ E(t):= \int_{{\mathbb R}^2} ( \sum_{\mu =0}^2 |D_{\mu} \phi(t)|^2 + V(|\phi(t)|^2)) \, dx \, . $$

This model was proposed by Hong, Kim and Pac \cite{HKP} and Jackiw and Weinberg \cite{JW} in the study of vortex solutions in the abelian Chern-Simons theory.

The equations are invariant under the gauge transformations
$$ A_{\mu} \rightarrow A'_{\mu} = A_{\mu} + \partial_{\mu} \chi \, , \, \phi \rightarrow \phi' = e^{i\chi} \phi \, , \, D_{\mu} \rightarrow D'_{\mu} = \partial_{\mu}-iA'_{\mu} \, . $$
The most common gauges are the Coulomb gauge $\partial^j A_j =0$ , the Lorenz gauge $\partial^{\mu} A_{\mu} = 0$ and the temporal gauge $A_0 = 0$. In this paper we exclusively study the temporal gauge for finite energy data.

Global well-posedness in the Coulomb gauge was proven by Chae and Choe \cite{CC} for data $a_{\mu} \in H^a$ , $\phi_0 \in H^b$ , $\phi_1 \in H^{b-1}$ where $(a,b)=(l,l+1)$ with $l \ge 1$ , satisfying a compatibility condition and a class of Higgs potentials. Huh \cite{H} showed local well-posedness in the Coulomb gauge for $(a,b)=(\epsilon,1+\epsilon)$ and in the Lorenz gauge for $(a,b) = (\frac{3}{4}+\epsilon,\frac{9}{8}+\epsilon)$ or $(a,b)=(\frac{1}{2},\frac{3}{2})$ , and in the temporal gauge for $(a,b)=(l,l)$ with $l \ge \frac{3}{2}$ .

He also showed global well-posedness in the temporal gauge for data $\phi_0 \in H^2$, $\phi_1 \in H^1$, $a_i^{df} \in H^1$, $a_i^{cf} \in H^2$, where $a^{df}$ and $a^{cf}$ denote the divergence-free and curl-free part of $a$. 

The local well-posedness result in the Lorenz gauge was improved to $(a,b)=(l,l+1)$ and $l > \frac{1}{4}$ by Bournaveas \cite{B} and by Yuan \cite{Y}.  Also in Lorenz gauge the important global well-posedness result in energy space, where $a_{\mu} \in \dot{H}^{\frac{1}{2}}$ , $\phi_0 \in H^1,$  $\phi_1 \in L^2$ , was proven by Selberg and Tesfahun \cite{ST} under a sign condition on the potential V, and even unconditional well-posedness could be proven by Selberg and Oliveira da Silva \cite{SO}. In \cite{ST} the regularity assumptions on the data could also be lowered down in Lorenz gauge to $(a,b)=(l,l+\frac{1}{2})$ and $l > \frac{3}{8}$ . This latter result was improved to $l > \frac{1}{4}$ by Huh and Oh \cite{HO}. Global well-posedness in energy space and local well-posedness for $a_{\mu} \in \dot{H}^{\frac{1}{2}}$ , $\phi_0 \in H^{l+\frac{1}{2}}$ , $\phi_1 \in H^{l-\frac{1}{2}}$ for $\frac{1}{2} \ge l > \frac{1}{4}$ in Coulomb gauge was recently obtained by Oh \cite{O}. For all these results up to the paper by Chae and Choe \cite{CC} and Oh \cite{O} it was crucial to make use  of a null condition in the nonlinearity of the system.

A low regularity local well-posedness result in the temporal gauge  for the Yang-Mills equations was given by Tao \cite{T1}.

In this paper we consider exclusively the temporal gauge. We show local well-posedness in energy space and above for potentials $V$ of polynomial growth, more precisely for data $\phi_0 \in H^1$ , $\phi_1 \in L^2$ ,  $|\nabla|^{\epsilon} a_j \in H^{\frac{1}{2}}$ ($\epsilon > 0$ small), under the compatibility assumption $\partial_1 a_2 - \partial_2 a_1 = 2 Im(\overline{\phi}_0 \phi_1)$.  If $V$ satisfies the sign condition $V(r) \ge -\alpha^2 r$ $\forall\,  r \ge 0$ , where $\alpha >0$, this solution exists globally in time. 

Thus we directly show global well-posedness for finite energy data in temporal gauge, wehich was known before in the case of the Coulomb and the Lorenz gauge.

We use a contraction argument in $X^{s,b}$ - type spaces adapted to the phase functions $\tau \pm |\xi|$ on one hand and to the phase function $\tau$ on the other hand. We also take advantage of a null condition which appears in the nonlinearity. Most of the crucial arguments follow from the bilinear estimates in wave-Sobolev spaces established by d'Ancona, Foschi and Selberg \cite{AFS}, which rely on the Strichartz estimates. Moreover we use an estimate for the $L^6_x L^2_t$-norm for the solution of the wave equation which goes back to Tataru \cite{KMBT} and Tao \cite{T1}. When applying this estimate we partly follow Tao's arguments in the case of the Yang-Mills equations. For the global existence part, which of course relies on energy conservation, we adapt the proof of Selberg-Tesfahun \cite{ST} for the Lorenz gauge to the temporal gauge.

We denote the Fourier transform with respect to space and time by $\,\widehat{}$ . The operator
$|\nabla|^{\alpha}$ is defined by $(|\nabla|^{\alpha} f)(\xi) = |\xi|^{\alpha} ({\mathcal F}f)(\xi)$, where ${\mathcal F}$ is the Fourier transform, and similarly $ \langle \nabla \rangle^{\alpha}$. The inhomogeneous and homogeneous Sobolev spaces are denoted by $H^{s,p}$ and $\dot{H}^{s,p}$, respectively. For $p=2$ we simply denote them by $H^s$ and $\dot{H}^s$. We repeatedly use the Sobolev embeddings $\dot{H}^{s,p} \subset L^q$ for  $1<p\le q < \infty$ and $\frac{1}{q} = \frac{1}{p}-\frac{s}{2}$, and also $\dot{H}^{1+} \cap \dot{H}^{1-} \subset L^{\infty}$ in two space dimensions. \\
$a+ := a + \epsilon$ for a sufficiently small $\epsilon >0$ , so that $a<a+<a++$ , and similarly $a--<a-<a$ , and $\langle \cdot \rangle := (1+|\cdot|^2)^{\frac{1}{2}}$ .

We now formulate our main results and begin by defining the standard spaces $X^{s,b}_{\pm}$ of Bourgain-Klainerman-Machedon type belonging to the half waves as the completion of the Schwarz space  $\mathcal{S}({\mathbb R}^3)$ with respect to the norm
$$ \|u\|_{X^{s,b}_{\pm}} = \| \langle \xi \rangle^s \langle  \tau \pm |\xi| \rangle^b \widehat{u}(\tau,\xi) \|_{L^2_{\tau \xi}} \, . $$ 
Similarly we define the wave-Sobolev spaces $X^{s,b}_{|\tau|=|\xi|}$ with norm
$$ \|u\|_{X^{s,b}_{|\tau|=|\xi|}} = \| \langle \xi \rangle^s \langle  |\tau| - |\xi| \rangle^b \widehat{u}(\tau,\xi) \|_{L^2_{\tau \xi}}  $$ and also $X^{s,b}_{\tau =0}$ with norm 
$$\|u\|_{X^{s,b}_{\tau=0}} = \| \langle \xi \rangle^s \langle  \tau  \rangle^b \widehat{u}(\tau,\xi) \|_{L^2_{\tau \xi}} \, .$$
We also define $X^{s,b}_{\pm}[0,T]$ as the space of the restrictions of functions in $X^{s,b}_{\pm}$ to $[0,T] \times \mathbb{R}^2$ and similarly $X^{s,b}_{|\tau| = |\xi|}[0,T]$ and $X^{s,b}_{\tau =0}[0,T]$. We frequently use the estimates $\|u\|_{X^{s,b}_{\pm}} \le \|u\|_{X^{s,b}_{|\tau|=|\xi|}}$ for $b \le 0$ and the reverse estimate for $b \ge 0$. 

Our main theorems read as follows:                                  
\begin{theorem}
\label{Theorem1.1}
Let $\epsilon > 0$ be sufficiently small, and $V \in C^{\infty}({\mathbb R}^+,{\mathbb R})$ , $V(0)=0$, and all derivatives have polynomial growth. The Chern-Simons-Higgs system (\ref{*1}),(\ref{*2}),(\ref{*3}) in temporal gauge $A_0=0$ with data $\phi_0 \in H^1({\mathbb R}^2)$ , $ \phi_1 \in L^2({\mathbb R}^2)$ , $|\nabla|^{\epsilon} a_j \in H^{\frac{1}{2}}({\mathbb R}^2) $ satisfying the compatibility condition $\partial_1 a_2 - \partial_2 a_1 = 2 Im(\overline{\phi}_0 \phi_1)$ has a unique local solution $$\phi \in C^0([0,T],H^1({\mathbb R}^2)) \cap C^1([0,T],L^2({\mathbb R}^2)) \, , \,  |\nabla|^{\epsilon} A \in C^0([0,T],H^{\frac{1}{2}}({\mathbb R}^2)) \,.  $$
More precisely, $T$ only depends on the data norm
$$\|\phi_0\|_{H^1} + \|\phi_1\|_{L^2} + \| |\nabla|^{\epsilon} A^{cf}(0)\|_{H^{\frac{1}{2}}} $$
and $\phi = \phi_+ + \phi_-$ with $\phi_{\pm} \in X_{\pm}^{1,\frac{1}{2}+\epsilon-}[0,T]$. If $A= A^{df} + A^{cf} =-(-\Delta)^{-1} \nabla \,div\, A + (-\Delta)^{-1} \,curl\, curl A$ is the decomposition into its divergence-free and its curl-free part, one has $ \nabla A^{cf} \in X^{{-\frac{1}{2}+\epsilon},\frac{1}{2}+}_{\tau=0} [0,T]$ , $|\nabla|^{\epsilon} A^{cf} \in C^0([0,T],L^2({\mathbb R}^2))$ , $\nabla A^{df} \in X^{{0},\frac{1}{2}-}_{|\tau|=|\xi|}[0,T],$ $|\nabla|^{\epsilon} A^{df} \in C^0([0,T],L^2({\mathbb R}^2))$, and in these spaces uniqueness holds. Moreover one has $\nabla A^{df} \in X^{-\delta,\frac{1}{2}+\delta-}_{|\tau| =|\xi|}[0,T]$ for $0<\delta<\epsilon$, and higher regularity persists. In particular the solution is smooth, if the data are smooth.
\end{theorem}

\begin{theorem}
\label{Theorem1.2}
Assume in addition that $V(r) \ge \alpha^2 r$ for all $r\ge 0$ for some $\alpha >0$. Then the solution of Theorem \ref{Theorem1.1} exists globally in time.
\end{theorem}
\noindent {\bf Remark:} 1. Under our assumptions on the data the energy $E(0)$ is finite. We namely have
$\|D_j \phi(0)\|_{L^2} \lesssim \| \partial_j \phi(0)\|_{L^2} + \|a_j \phi_0\|_{L^2} \lesssim \|\phi_0\|_{H^1} + \|\nabla|^{\epsilon} a_j\|_{H^{\frac{1}{2}}} \|\phi_0\|_{H^1} < \infty $
and $V(|\phi|^2) \in L^1$, because $H^1 \subset L^p$ for all $2 \le p < \infty$. \\
2. Persistence of higher regularity is a standard fact for solutions constructed by a Picard iteration, so we omit its proof.

\section{Reformulation of the problem}
In the temporal gauge $A_0 =0$ the Chern-Simons-Higgs system
 (\ref{*1}),(\ref{*2}) is equivalent to the following system
\begin{align}
\label{**1}
&\partial_t A_j   = 2 \epsilon_{ij} Im(\overline{\phi} D^i \phi) \\
\label{**2}
&\partial_t^2 \phi - D^j D_j \phi  = - \phi V'(|\phi|^2) \\
\nonumber
&\Longleftrightarrow \Box \phi  = 2i A^j \partial_j \phi - i \partial_j A^j \phi + A^j A_j \phi - \phi V'(|\phi|^2) \\
\label{1.9}
&\partial_1 A_2 - \partial_2 A_1  = 2 Im(\overline{\phi} \partial_t \phi) \, ,
\end{align}
where $i,j=1,2$ , $\epsilon_{12}=1$ , $\epsilon_{21}=-1$ and $\Box = \partial_t^2 - \partial_1^2 - \partial_2^2$ . \\
We remark that (\ref{1.9}) is fulfilled for any solution of (\ref{**1}),(\ref{**2}), if it holds initially, i.e. , if the following compatibility condition holds, which we assume from now on:
\begin{equation}
\label{1.9init}
\partial_1 A_2(0) - \partial_2 A_1(0) = 2 Im(\overline{\phi}(0) (\partial_t \phi)(0)) \, .
\end{equation}
Indeed, we have by (\ref{**1}) and (\ref{**2}):
$$ \partial_t(\partial_1 A_2 - \partial_2 A_1) = 2 Im(\overline{\phi}(D_1^2 \phi + D_2^2 \phi)) = 2 Im(\overline{\phi} \partial_t^2 \phi) = 2 \partial_t Im(\overline{\phi} \partial_t \phi) \, . $$
Thus we only have to solve (\ref{**1}) and (\ref{**2}), and can assume that (\ref{1.9}) is fulfilled. We make the standard decomposition of $A=(A_1,A_2)$ into its divergence-free part $A^{df}$ and its curl-free part $A^{cf}$, namely $A=A^{df} +A^{cf}$, where
\begin{align*}
A^{df} & = (-\Delta)^{-1}(\partial_1 \partial_2 A_2 - \partial_2^2 A_1,\partial_1 \partial_2 A_1 - \partial_1^2 A_2) = (-\Delta)^{-1} \,curl\, curl A\, , \\
A^{cf} & = -(-\Delta)^{-1}(\partial_1 \partial_2 A_2 + \partial_1^2 A_1, \partial_1 \partial_2 A_1 + \partial_2^2 A_2) = -(-\Delta)^{-1} \nabla \,div\, A \, .
\end{align*}
Let $B$ be defined by $A^{df}_1 = -\partial_2 B$ , $A^{df}_2 = \partial_1 B$ . Then by (\ref{1.9}) and $\partial_1 A^{cf}_2 - \partial_2 A^{cf}_1 =0$ we obtain
$$ \Delta B = \partial_1 A_2^{df} - \partial_2 A_1^{df} = 2 Im(\overline{\phi} \partial_t \phi) \, ,$$
so that
\begin{equation}
\label{***2}
A^{df}_1 = -2 \Delta^{-1} \partial_2 Im(\overline{\phi} \partial_t \phi) \, , \, A^{df}_2 = 2 \Delta^{-1} \partial_1 Im(\overline{\phi} \partial_t \phi) \, .
\end{equation}
Next we calculate $\partial_t A^{cf}$ for solutions $(A,\phi)$ of (\ref{**1}),(\ref{**2}):
\begin{align}
\nonumber
\partial_t A^{cf}_1 &= \Delta^{-1} \partial_1(\partial_2 \partial_t A_2 + \partial_1 \partial_t A_1) \\
\nonumber
& = 2 \Delta^{-1} \partial_1(\partial_2 Im(\overline{\phi} D_1 \phi) - \partial_1 Im(\overline{\phi} D_2 \phi) \\
\nonumber
& = 2 \Delta^{-1} \partial_1 Im(\partial_2 \overline{\phi} D_1 \phi + \overline{\phi} \partial_2 D_1 \phi - \partial_1 \overline{\phi} D_2 \phi - \overline{\phi} \partial_1 D_2 \phi) \\
\nonumber
& = 2 \Delta^{-1} \partial_1 Im [\partial_2 \overline{\phi}\partial_1 \phi - \partial_1 \overline{\phi} \partial_2 \phi -i A_1 \partial_2 \overline{\phi} \phi + i A_2 \partial_1 \overline{\phi} \phi \\
\nonumber
& \quad + \overline{\phi}(-i A_1 \partial_2 \phi + i A_2 \partial_1 \phi -i \partial_2 A_1 \phi + i \partial_1 A_2 \phi) ] \\
\nonumber
& = 2 \Delta^{-1} \partial_1 Im [\partial_2 \overline{\phi} \partial_1 \phi - \partial_1 \overline{\phi} \partial_2 \phi + i \overline{\phi}(A_2 \partial_1 \phi -A_1 \partial_2 \phi) \\
\nonumber
&  \quad + i \phi(A_2 \partial_1 \overline{\phi} - A_1 \partial_2 \overline{\phi}) + i(\partial_1 A_2 - \partial_2 A_1) |\phi|^2 ] \\
\nonumber
& = 2 \Delta^{-1} \partial_1 Im(\partial_2 \overline{\phi} \partial_1 \phi - \partial_1 \overline{\phi} \partial_2 \phi)  + 2 \Delta^{-1} \partial_1(A_2 \partial_1 |\phi|^2 -A_1 \partial_2 |\phi|^2) \\
\label{***3a}
& \quad + 4 \Delta^{-1} \partial_1 Im(\overline{\phi} \partial_t \phi) |\phi|^2 \, .
\end{align}
Similarly
\begin{align}
\nonumber
 \partial_t A^{cf}_2 & = 2 \Delta^{-1} \partial_2 Im(\partial_1 \overline{\phi} \partial_2 \phi - \partial_2 \overline{\phi} \partial_1 \phi)  + 2 \Delta^{-1} \partial_2(A_1 \partial_2 |\phi|^2 -A_2 \partial_1 |\phi|^2) \\
\label{***3b}
& \quad + 4 \Delta^{-1} \partial_2 Im(\overline{\phi} \partial_t \phi) |\phi|^2 \, .
\end{align}

Moreover from (\ref{**2}) we obtain using $\partial^j A_j^{df} =0$:
\begin{equation}
\label{***1}
\Box \phi = 2i A^{cf} \nabla \phi +2i A^{df} \nabla \phi -i \partial^j A_j^{cf} \phi + (A^{df,j} + A^{cf,j})(A^{df}_j + A^{cf}_j) \phi - \phi V'(|\phi|^2)  .
\end{equation}

We also obtain from (\ref{***2}) and (\ref{**2})
$$ \partial_t A^{df}_2 = 2 \partial_t \Delta^{-1} \partial_1 Im(\overline{\phi} \partial_t \phi) =2 \Delta^{-1} \partial_1 Im(\overline{\phi} \partial_t^2 \phi) = 2\Delta^{-1} \partial_1 Im(\overline{\phi} D^j D_j \phi) \, . $$
Now
\begin{align*}
Im(\overline{\phi} D^j D_j \phi) & = Im(\overline{\phi} \partial^j \partial_j \phi - 2i \overline{\phi} A_j \partial^j \phi - i \overline{\phi} \partial^j A_j \phi) \\
& = Im(\overline{\phi} \partial^j \partial_j \phi - i \overline{\phi} A_j \partial^j \phi -i A_j \phi \partial^j \overline{\phi}- i \overline{\phi} \partial^j A_j \phi) \\
& = \partial^j Im(\overline{\phi} D_j \phi) \, ,
\end{align*}
so that
\begin{equation}
\label{****1}
\partial_t A^{df}_2 = 2\Delta^{-1} \partial_1 \partial^j Im(\overline{\phi} D_j \phi) \, .
\end{equation}
Similarly we obtain
\begin{equation}
\label{****2}
\partial_t A^{df}_1 = - 2\Delta^{-1} \partial_2 \partial^j Im(\overline{\phi} D_j \phi) \, .
\end{equation}
Reversely defining $A:= A^{df} + A^{cf}$ we show that our new system (\ref{***2}),(\ref{***3a}),(\ref{***3b}),(\ref{***1}) implies (\ref{**1}),(\ref{**2}) and also (\ref{1.9}), provided the compatability condition (\ref{1.9init}) is fulfilled. (\ref{**2}) is obvious. (\ref{1.9}) is fulfilled because by use of (\ref{***3a}) and (\ref{***3b})  one easily checks $\partial_1 A^{cf}_2 - \partial_2 A^{cf}_1 = 0$, so that by (\ref{***2})
$$ \partial_t (\partial_1 A_2 - \partial_2 A_1) = \partial_t( \partial_1 A^{df}_2 - \partial_2 A^{df}_1) = \partial_t Im(\overline{\phi} \partial_t \phi) \, . $$
Thus (\ref{1.9}) is fulfilled, if (\ref{1.9init}) holds. Finally we obtain
\begin{align*}
&\partial_t A_1  = \partial_t A_1^{cf} + \partial_t A^{df}_1 \\
& = 2 \Delta^{-1} \partial_1(\partial_2 Im(\overline{\phi} D_1 \phi) - \partial_1 Im(\overline{\phi} D_2 \phi))\hspace{-0.1em} - \hspace{-0.1em} 2 \Delta^{-1} \partial_2(\partial_1 Im(\overline{\phi} D_1 \phi) + \partial_2 Im(\overline{\phi} D_2 \phi)) \\
& = - 2 Im(\overline{\phi} D_2 \phi) \, ,
\end{align*}
where we used (\ref{***3a}) and also (\ref{****2}), which was shown to be a consequence of (\ref{***2}) and (\ref{***1}).
Similarly we also get
$$ \partial _t A_2 = 2 Im(\overline{\phi} D_1 \phi) \, , $$
so that (\ref{**1}) is shown to be satisfied.

Summarizing we have shown that (\ref{**1}),(\ref{**2}),(\ref{1.9}) are equivalent to (\ref{***2}),(\ref{***3a}),(\ref{***3b}),(\ref{***1}) (which also implies (\ref{****1}),(\ref{****2})).

Concerning the initial conditions assume we are given initial data for our system (\ref{**1}),(\ref{**2}),(\ref{1.9}):
$$ A_j(0) = a_j \, , \, \phi(0) = \phi_0 \, , \, (\partial_t \phi)(0) = \phi_1 $$
satisfying $|\nabla|^{\epsilon}a_j \in H^{\frac{1}{2}}$ , $\phi_0 \in H^1 $ , $ \phi_1 \in L^2$ and (\ref{1.9init}). Then by (\ref{***2}) and (\ref{1.9init}) we obtain
\begin{align*}
A^{df}_1(0) &= -2 \Delta^{-1} \partial_2 Im(\overline{\phi}_0 \phi_1) = - \Delta^{-1} \partial_2(\partial_1 a_2 - \partial_2 a_1)  \\
A^{df}_2(0)& = 2 \Delta^{-1} \partial_1 Im(\overline{\phi}_0 \phi_1) =  \Delta^{-1} \partial_1(\partial_1 a_2 - \partial_2 a_1) 
\end{align*}
and 
$$ A^{cf}_j(0) = a_j - A^{df}_j(0)  ,$$
thus $$|\nabla|^{\epsilon} A^{df}_j(0)\in H^{\frac{1}{2}} \, , \, |\nabla|^{\epsilon}A^{cf}_j(0) \in H^{\frac{1}{2}}\, .$$
In the sequel we construct a solution of the Cauchy problem for (\ref{***2}),(\ref{***3a}),(\ref{***3b}), (\ref{***1}) with data $\phi_0 \in H^1$ , $ \phi_1 \in L^2$ , $|\nabla|^{\epsilon} A^{cf}_j(0) \in H^{\frac{1}{2}}$. We have shown that whenever we have a local solution of this system with data $\phi_0,\phi_1$ and $A_j^{cf}(0) = a_j - A^{df}_j (0)$, where $A^{df}_1(0) = -2 \Delta^{-1} \partial_2 Im(\overline{\phi}_0 \phi_1)$ , 
$A^{df}_2(0) = 2 \Delta^{-1} \partial_1 Im(\overline{\phi}_0 \phi_1)$, we also have that $(\phi,A)$ with $A:= A^{df} + A^{cf}$ is a local solution of (\ref{**1}),(\ref{**2}) with data $(\phi_0,\phi_1,a_1,a_2)$. If (\ref{1.9init}) holds then (\ref{1.9}) is also satisfied.

Defining
$$ \phi_{\pm} = \frac{1}{2}(\phi \pm i^{-1} \langle \nabla \rangle^{-1} \partial_t \phi) \, \Longleftrightarrow \, \phi=\phi_+ +\phi_- \, , \, \partial_t \phi= i \langle \nabla \rangle (\phi_+ - \phi_-) $$
the equation (\ref{***1}) transforms to
\begin{align}
\label{***1'}
(i \partial_t \pm \langle \nabla \rangle) \phi_{\pm} 
 = &\pm 2^{-1} \langle \nabla \rangle^{-1} \big(2i A^{cf} \nabla \phi +2i A^{df} \nabla \phi -i \partial^j A_j^{cf} \phi \\
 \nonumber
 &+ (A^{df,j} + A^{cf,j})(A^{df}_j + A^{cf}_j) \phi - \phi V'(|\phi|^2) + \phi \big)
\end{align}

Fundamental for the proof of our theorem are the following bilinear estimates in wave-Sobolev spaces which were proven by d'Ancona, Foschi and Selberg in the two-dimensional case $n=2$ in \cite{AFS} in a more general form which include many limit cases which we do not need.
\begin{theorem}
\label{Theorem3}
Let $n=2$. The estimate
$$\|uv\|_{X_{|\tau|=|\xi|}^{-s_0,-b_0}} \lesssim \|u\|_{X^{s_1,b_1}_{|\tau|=|\xi|}} \|v\|_{X^{s_2,b_2}_{|\tau|=|\xi|}} $$ 
holds, provided the following conditions hold:
\begin{align*}
\nonumber
& b_0 + b_1 + b_2 > \frac{1}{2} \\
\nonumber
& b_0 + b_1 \ge 0 \\
\nonumber
& b_0 + b_2 \ge 0 \\
\nonumber
& b_1 + b_2 \ge 0 \\
\nonumber
&s_0+s_1+s_2 > \frac{3}{2} -(b_0+b_1+b_2) \\
\nonumber
&s_0+s_1+s_2 > 1 -\min(b_0+b_1,b_0+b_2,b_1+b_2) \\
\nonumber
&s_0+s_1+s_2 > \frac{1}{2} - \min(b_0,b_1,b_2) \\
\nonumber
&s_0+s_1+s_2 > \frac{3}{4} \\
 &(s_0 + b_0) +2s_1 + 2s_2 > 1 \\
\nonumber
&2s_0+(s_1+b_1)+2s_2 > 1 \\
\nonumber
&2s_0+2s_1+(s_2+b_2) > 1 \\
\nonumber
&s_1 + s_2 \ge \max(0,-b_0) \\
\nonumber
&s_0 + s_2 \ge \max(0,-b_1) \\
\nonumber
&s_0 + s_1 \ge \max(0,-b_2)   \, .
\end{align*}
\end{theorem}

We also need the following
\begin{prop}
The following estimates hold
\begin{align}
\label{Str}
\|u\|_{L^6_{xt}} & \lesssim \|u\|_{X^{\frac{1}{2},\frac{1}{2}+}_{|\tau|=|\xi|}} \, , \\
\label{T}
\|u\|_{L^6_x L^2_t} & \lesssim \|u\|_{X^{\frac{1}{6},\frac{1}{2}+}_{|\tau|=|\xi|}} \, .
\end{align}
\end{prop}
\begin{proof}
(\ref{Str}) is the original Strichartz estimate \cite{Str} combined with the transfer principle. (\ref{T}) goes back to \cite{KMBT}, Thm. 3.2:
$$ \|{\mathcal F}_t u \|_{L^2_{\tau} L^6_x} \lesssim \|u_0\|_{\dot{H}^{\frac{1}{6}}} \, , $$
if $u= e^{it|\nabla|} u_0$ and ${\mathcal F}_t$ denotes the Fourier transform with respect to time. By Plancherel and Minkowski's inequality we obtain
$$\|u\|_{L^6_x L^2_t} = \|{\mathcal F}_t u\|_{L^6_x L^2_{\tau}} \lesssim \|{\mathcal F}_t u\|_{L^2_{\tau} L^6_x} \lesssim \|u_0\|_{\dot{H}^{\frac{1}{6}}} \, . $$
The transfer principle gives (\ref{T}).
\end{proof}
The following easy consequences are obtained by interpolation between (\ref{Str}), (\ref{T}) and the trivial identity
\begin{align}
\label{I0}
\|u\|_{L^2_{xt}} &= \|u\|_{X^{0,0}_{|\tau|=|\xi|}} \, : \\
\label{I1}
\|u\|_{L^6_x L^{2+}_t} & \lesssim \|u\|_{X^{\frac{1}{6}+,\frac{1}{2}+}_{|\tau|=|\xi|}} & \mbox{(interpolate (\ref{Str}) and (\ref{T}))} \, , \\
\label{I2}
\|u\|_{L^4_x L^4_t} & \lesssim \|u\|_{X^{\frac{3}{8},\frac{3}{8}+}_{|\tau|=|\xi|}} & \mbox{(interpolate (\ref{Str}) and (\ref{I0}))} \, , \\
\label{I3}
\|u\|_{L^4_x L^{2+}_t} & \lesssim \|u\|_{X^{\frac{1}{8}+,\frac{3}{8}+}_{|\tau|=|\xi|}} & \mbox{(interpolate (\ref{I1}) and (\ref{I0}))} \, , \\
\label{I4}
\|u\|_{L^{4+}_x L^{2+}_t} & \lesssim \|u\|_{X^{\frac{1}{8}+,\frac{3}{8}+}_{|\tau|=|\xi|}} & \mbox{(interpolate (\ref{I1}) and (\ref{I0}))} \, , \\
\label{I5}
\|u\|_{L^3_x L^{2}_t} & \lesssim \|u\|_{X^{\frac{1}{12},\frac{1}{4}+}_{|\tau|=|\xi|}} & \mbox{(interpolate (\ref{T}) and (\ref{I0}))} \, , \\
\label{I6}
\|u\|_{L^{\frac{2}{1-\epsilon}}_x L^{2}_t} & \lesssim \|u\|_{X^{\frac{\epsilon}{4},\frac{3}{4}\epsilon+}_{|\tau|=|\xi|}} & \mbox{(interpolate (\ref{T}) and (\ref{I0}))} \,.
\end{align}

\section{Proof of Theorem \ref{Theorem1.1}}
Taking the considerations of the previous section into account Theorem \ref{Theorem1.1} reduces to the following proposition and its corollary.
\begin{prop}
Let $\epsilon > 0$ be sufficiently small.
The system 
\begin{align}
\label{***1''}
(i \partial_t \pm \langle \nabla \rangle) \phi_{\pm} &
 = \pm 2^{-1} \langle \nabla \rangle^{-1} \big(2i A^{cf} \nabla \phi +2i A^{df} \nabla \phi -i\, \partial^j A_j^{cf} \phi \\
 \nonumber
 & \quad + (A^{df,j} + A^{cf,j})(A^{df}_j + A^{cf}_j) \phi - \phi V'(|\phi|^2) + \phi \big) \\
\label{***2'}
A^{df}_1 &= -2 \Delta^{-1} \partial_2 Im(\overline{\phi} \partial_t \phi) \quad , \quad
 A^{df}_2 = 2 \Delta^{-1} \partial_1 Im(\overline{\phi} \partial_t \phi) \\ 
 \nonumber
\partial_t A^{cf}_1 &= 2 \Delta^{-1} \partial_1 Im(\partial_2 \overline{\phi} \partial_1 \phi - \partial_1 \overline{\phi} \partial_2 \phi)  + 2 \Delta^{-1} \partial_1(A_2 \partial_1 |\phi|^2 -A_1 \partial_2 |\phi|^2) \\
\label{***3a'} 
& \quad + 4 \Delta^{-1} \partial_1 Im(\overline{\phi} \partial_t \phi) |\phi|^2 \, \\
\nonumber
 \partial_t A^{cf}_2 & = 2 \Delta^{-1} \partial_2 Im(\partial_1 \overline{\phi} \partial_2 \phi - \partial_2 \overline{\phi} \partial_1 \phi)  + 2 \Delta^{-1} \partial_2(A_1 \partial_2 |\phi|^2 -A_2 \partial_1 |\phi|^2) \\
\label{***3b'}
& \quad + 4 \Delta^{-1} \partial_2 Im(\overline{\phi} \partial_t \phi) |\phi|^2 \, . 
\end{align}
with data
$ \phi_{\pm}(0) \in H^{1} $ and $|\nabla|^{\epsilon}A^{cf}(0) \in H^{\frac{1}{2}}$ has a unique local solution $$\phi_{\pm} \in X^{1,\frac{1}{2}+\epsilon-}_{\pm}[0,T] \, , \, \nabla A^{cf} \in X^{-\frac{1}{2}+\epsilon,\frac{1}{2}+}_{\tau =0}[0,T] \, , \, |\nabla|^{\epsilon} A^{cf} \in C^0([0,T],L^2) \, .$$ 
Here $\phi=\phi_+ +\phi_- \, , \, \partial_t \phi= i \langle \nabla \rangle (\phi_+ - \phi_-)$ .
Moreover $A^{df}$ satisfies $\nabla A^{df} \in X^{0,\frac{1}{2}-}_{|\tau|=|\xi|}[0,T]$ , $|\nabla|^{\epsilon}A^{df} \in C^0([0,T],L^2)$ and  also  $\nabla A^{df} \in X^{-\delta,\frac {1}{2}+\delta-}_{|\tau| =|\xi|}[0,T]$ for $0<\delta<\epsilon$.
\end{prop}
We obtain immediately
\begin{Cor}
The solution has the property
 $ \, \phi \in C^0([0,T],H^1) \cap C^1([0,T],L^2),$  $|\nabla|^{\epsilon} A^{cf} \in C^0([0,T],H^{\frac{1}{2}})$ and $|\nabla|^{\epsilon} A^{df} \in C^0([0,T],H^{1-\epsilon-}) \, . $
\end{Cor}
\begin{proof}
We want to apply the contraction mapping principle for
$$\phi_{\pm} \in X_{\pm}^{1,\frac{1}{2}+\epsilon-}[0,T] \, , \nabla A^{cf} \in X_{\tau =0}^{-\frac{1}{2}+\epsilon,\frac{1}{2}+}[0,T] \, , \, |\nabla|^{\epsilon} A^{cf} \in C^0([0,T],L^2) \, . $$
By well-known arguments this is reduced to the estimates
of the right hand sides of (\ref{***1''}),(\ref{***3a'}) and (\ref{***3b'}) stated as claims 1-9 below. We start to control $\nabla A^{cf}$ in $ X^{-\frac{1}{2}+\epsilon,\frac{1}{2}+}_{\tau =0}$.\\
{\bf Claim 1:} $$ \|\partial_i \overline{\phi} \partial_j \phi - \partial_j \overline{\phi} \partial_i \phi \|_{X^{-\frac{1}{2}+\epsilon,-\frac{1}{2}+}_{\tau =0}} \lesssim \|\nabla \phi\|^2_{X^{0,\frac{1}{2}+\epsilon-}_{|\tau| = |\xi|}} \, . $$
Let $\pm_1$ and $\pm_2$ denote independent signs. Using
$$\partial_i \overline{\phi} \partial_j \phi - \partial_j \overline{\phi} \partial_i \phi  = \sum_{\pm_1,\pm_2} (\partial_i \overline{\phi}_{\pm_1} \partial_j \phi_{\pm_2} - \partial_j \overline{\phi}_{\pm_1} \partial_i \phi_{\pm_2})$$ 
it suffices to show 
$$ \|\partial_i \overline{\phi} \partial_j \psi - \partial_j \overline{\phi} \partial_i \psi \|_{X^{-\frac{1}{2}+\epsilon,-\frac{1}{2}+}_{\tau =0}} \lesssim \|\nabla \phi\|_{X^{0,\frac{1}{2}+\epsilon-}_{\pm_1}} \|\nabla \psi\|_{X^{s,\frac{1}{2}+\epsilon-}_{\pm_2}} \, . $$
We now use the null structure of this term in the form that for vectors $\xi=(\xi^1,\xi^2),$  $\eta=(\eta^1,\eta^2) \in {\mathbb R}^2$ the following estimate holds
$$ |\xi^i \eta^j - \xi^j \eta^i| \le |\xi| |\eta| \angle(\xi,\eta) \, , $$
where $\angle(\xi,\eta)$ denotes the angle between $\xi$ and $\eta$. The following lemma gives the decisive bound for the angle:
\begin{lemma} (\cite{S}, Lemma 2.1 or \cite{ST}, Lemma 3.2)
\begin{equation}
\label{angle}
\angle(\pm_1 \xi_1,\pm_2 \xi_2) \lesssim \Big(\frac{\langle \tau_1 \pm_1 |\xi_1| \rangle + \langle \tau_2 \pm_2 |\xi_2| \rangle}{\min(\langle \xi_1 \rangle,\langle \xi_2 \rangle)}\Big)^{\frac{1}{2}} 
+\Big( \frac{\langle |\tau_3|-|\xi_3| \rangle}{\min(\langle \xi_1 \rangle,\langle \xi_2 \rangle)}\Big)^{\frac{1}{2}-\epsilon}
\end{equation}
$\forall \, \xi_1,\xi_2,\xi_3 \in{\mathbb R}^2 \, , \, \tau_1,\tau_2,\tau_3 \in {\mathbb R}$ with $\xi_1+\xi_2+\xi_3=0$ and $\tau_1+\tau_2+\tau_3 =0$.
\end{lemma}
\noindent Thus the claimed estimate reduces to 
\begin{align}
\label{20}
\Big| \int_* \frac{\widehat{u}_1(\tau_1,\xi_1)}{\langle \tau_1 \pm_1 |\xi_1|\rangle^{\frac{1}{2}+\epsilon-}} 
\frac{\widehat{u}_2(\tau_2,\xi_2)}{\langle \tau_2 \pm_2 |\xi_2|\rangle^{\frac{1}{2}+\epsilon-}}
\frac{\widehat{u}_3(\tau_3,\xi_3)}{\langle \xi_3 \rangle^{\frac{1}{2}-\epsilon}\langle \tau_3 \rangle^{\frac{1}{2}-}} \angle(\pm_1 \xi_1,\pm_2 \xi_2) \Big| 
\lesssim \prod_{i=1}^3 \|u_i\|_{L^2_{xt}} \, ,
\end{align}
where * denotes integration over $\xi_1,\xi_2,\xi_3,\tau_1,\tau_2,\tau_3$ with $\xi_1+\xi_2+\xi_3=0$ and $\tau_1+\tau_2+\tau_3 =0$. We assume without loss of generality that $|\xi_1| \le |\xi_2|$ and the Fourier transforms are nonnnegative. We distinguish three cases according to which of the terms on the right hand side of (\ref{angle}) is dominant. \\
Case 1: The last term in (\ref{angle}) dominant. 
In this case (\ref{20}) reduces to
\begin{align*}
\Big| \int_* \frac{\widehat{u}_1(\tau_1,\xi_1)}{\langle |\tau_1| - |\xi_1|\rangle^{\frac{1}{2}+\epsilon-} \langle \xi_1 \rangle^{\frac{1}{2}-\epsilon}} 
\frac{\widehat{u}_2(\tau_2,\xi_2)}{\langle |\tau_2| - |\xi_2|\rangle^{\frac{1}{2}+\epsilon-}}
\frac{\widehat{u}_3(\tau_3,\xi_3)}{\langle \xi_3 \rangle^{\frac{1}{2}-\epsilon} \langle \tau_3 \rangle^{\frac{1}{2}-}} \langle |\tau_3|-|\xi_3 | \rangle^{\frac{1}{2}-\epsilon} \Big| \\
\lesssim \prod_{i=1}^3 \|u_i\|_{L^2_{xt}} \, .
\end{align*}
1.1: $|\tau_3| \ge \frac{|\xi_3|}{2}$. In this case
(\ref{20}) reduces to
\begin{align*}
\Big| \int_* \frac{\widehat{u}_1(\tau_1,\xi_1)}{\langle |\tau_1| - |\xi_1|\rangle^{\frac{1}{2}+\epsilon-} \langle \xi_1 \rangle^{\frac{1}{2}-\epsilon}} 
\frac{\widehat{u}_2(\tau_2,\xi_2)}{\langle |\tau_2| - |\xi_2|\rangle^{\frac{1}{2}+\epsilon-}}
\frac{\widehat{u}_3(\tau_3,\xi_3)}{\langle \xi_3 \rangle^{1-\epsilon-}} \langle |\tau_3|-|\xi_3 | \rangle^{\frac{1}{2}-\epsilon} \Big|
 \\
\lesssim \prod_{i=1}^3 \|u_i\|_{L^2_{xt}} \, ,
\end{align*}
which follows from Theorem \ref{Theorem3}. \\
1.2: $|\tau_3| \le \frac{|\xi_3|}{2}$ $\Rightarrow$ $\langle |\tau_3|-|\xi_3|\rangle \sim \langle \xi_3 \rangle$. \\
1.2.1: $|\tau_2| \ll |\xi_2|$.
We have to show
\begin{align*}
\Big| \int_* \frac{\widehat{u}_1(\tau_1,\xi_1)}{\langle |\tau_1| - |\xi_1|\rangle^{\frac{1}{2}+\epsilon-} \langle \xi_1 \rangle^{\frac{1}{2}-\epsilon}} 
\frac{\widehat{u}_2(\tau_2,\xi_2)}{ |\xi_2|^{\frac{1}{2}+\epsilon-}}
\frac{\widehat{u}_3(\tau_3,\xi_3)}{\langle \tau_3 \rangle^{\frac{1}{2}-}} \Big|
 \lesssim \prod_{i=1}^3 \|u_i\|_{L^2_{xt}} \, .
\end{align*}
But by (\ref{I3}) we obtain
\begin{align*}
\Big| \int v_1 v_2 v_3 dx dt \Big| & \lesssim \|v_1\|_{L^4_x L^{2+}_t} 
\|v_2\|_{L^4_x L^2_t} \|v_3\|_{L^2_x L^{\infty-}_t} \\
&\lesssim \|v_1\|_{X^{\frac{1}{8}+,\frac{3}{8}+}_{|\tau|=|\xi|}} \|v_2\|_{H^{\frac{1}{2}}_x L^2_t} \|v_3\|_{L^2_x H^{\frac{1}{2}-}_t} \\
&\lesssim \|v_1\|_{X^{\frac{1}{2}+\epsilon-,\frac{1}{2}-\epsilon}_{|\tau|=|\xi|}} \|v_2\|_{X^{\frac{1}{2}+\epsilon-,0}_{|\tau|=|\xi|}    } \|v_3\|_{X^{0,\frac{1}{2}-}_{\tau =0}} \, , \\
\end{align*}
as desired.\\
1.2.2: $|\tau_2| \gtrsim |\xi_2|$. 
In this case we use $\tau_1 + \tau_2 + \tau_3 =0$ to estimate
$$ 1 \lesssim \frac{\langle \tau_2 \rangle^{\frac{1}{4}}}{\langle \xi_2 \rangle^{\frac{1}{4}}} \lesssim \frac{\langle \tau_1 \rangle^{\frac{1}{4}}}{\langle \xi_2 \rangle^{\frac{1}{4}}} + \frac{\langle \tau_3 \rangle^{\frac{1}{4}}}{\langle \xi_2 \rangle^{\frac{1}{4}}} \, . $$
The second term on the right hand side is taken care of by
\begin{align*}
\Big| \int_* \frac{\widehat{u}_1(\tau_1,\xi_1)}{\langle |\tau_1| - |\xi_1|\rangle^{\frac{1}{2}+\epsilon-} \langle \xi_1 \rangle^{\frac{1}{2}-\epsilon}} 
\frac{\widehat{u}_2(\tau_2,\xi_2)}{\langle \xi_2\rangle^{\frac{1}{4}} \langle |\tau_2|-|\xi_2|\rangle^{\frac{1}{2}+\epsilon-}}
\frac{\widehat{u}_3(\tau_3,\xi_3)}{|\tau_3|^{\frac{1}{4}-}} \Big|
\lesssim\prod_{i=1}^3 \|u_i\|_{L^2_{xt}} \, .
\end{align*}
This follows from
\begin{align*}
\Big| \int v_1 v_2 v_3 dx dt \Big| & \lesssim \|v_1\|_{L^4_x L^4_t} 
\|v_2\|_{L^4_x L^{2+}_t} \|v_3\|_{L^2_x L^{4-}_t} \\
&\lesssim \|v_1\|_{X^{\frac{3}{8},\frac{3}{8}+}_{|\tau|=|\xi|}} \|v_2\|_{X^{\frac{1}{8}+,\frac{3}{8}+}_{|\tau|=|\xi|}    } \|v_3\|_{X^{0,\frac{1}{4}-}_{\tau =0}} \, , \\
\end{align*}
which holds by (\ref{I2}) and (\ref{I3}).\\
For the first term we consider two subcases.\\
1.2.2.1: $|\tau_1| \lesssim |\xi_1|$ $\Rightarrow$ $\frac{\langle \tau_1 \rangle^{\frac{1}{4}}}{\langle \xi_2 \rangle^{\frac{1}{4}}} \lesssim \frac{\langle \xi_1 \rangle^{\frac{1}{4}}}{\langle \xi_2 \rangle^{\frac{1}{4}}}$.\\
We have to show
\begin{align*}
\Big| \int_* \frac{\widehat{u}_1(\tau_1,\xi_1)}{\langle |\tau_1| - |\xi_1|\rangle^{\frac{1}{2}+\epsilon-} \langle \xi_1 \rangle^{\frac{1}{4}-\epsilon}} 
\frac{\widehat{u}_2(\tau_2,\xi_2)}{\langle 
\xi_2\rangle^{\frac{1}{4}} \langle |\tau_2|-|\xi_2|\rangle^{\frac{1}{2}+\epsilon-}}
\frac{\widehat{u}_3(\tau_3,\xi_3)}{\langle \tau_3 \rangle^{\frac{1}{4}-}} \Big|
\lesssim \prod_{i=1}^3 \|u_i\|_{L^2_{xt}} \, .
\end{align*}
This follows from (\ref{I3}) which gives
\begin{align*}
\Big| \int v_1 v_2 v_3 dx dt \Big| & \lesssim \|v_1\|_{L^4_x L^{2+}_t} 
\|v_2\|_{L^4_x L^{2+}_t} \|v_3\|_{L^2_x L^{4-}_t} \\
&\lesssim \|v_1\|_{X^{\frac{1}{8}+,\frac{3}{8}+}_{|\tau|=|\xi|}} \|v_2\|_{X^{\frac{1}{8}+,\frac{3}{8}+}_{|\tau|=|\xi|}    } \|v_3\|_{X^{0,\frac{1}{4}-}_{\tau =0}} \, . \\
\end{align*}
1.2.2.2: $|\tau_1| \gg |\xi_1|$ $\Rightarrow $ $\frac{\langle \tau_1 \rangle^{\frac{1}{4}}}{\langle \xi_2 \rangle^{\frac{1}{4}}} \sim \frac{\langle |\tau_1|-| \xi_1| \rangle^{\frac{1}{4}}}{\langle \xi_2 \rangle^{\frac{1}{4}}}$.\\
Thus we need
\begin{align*}
\Big| \int_* \frac{\widehat{u}_1(\tau_1,\xi_1)}{\langle |\tau_1| - |\xi_1|\rangle^{\frac{1}{2}+\epsilon-} \langle \xi_1 \rangle^{\frac{1}{2}-\epsilon}} 
\frac{\widehat{u}_2(\tau_2,\xi_2)}{\langle \xi_2\rangle^{\frac{1}{4}} \langle |\tau_2|-|\xi_2|\rangle^{\frac{1}{2}+\epsilon-}}
\frac{\widehat{u}_3(\tau_3,\xi_3)}{\langle \tau_3 \rangle^{\frac{1}{2}-}} \Big|
\lesssim \prod_{i=1}^3 \|u_i\|_{L^2_{xt}} \, .
\end{align*}
But we have by (\ref{I5}) and (\ref{I1})
\begin{align*}
\Big| \int v_1 v_2 v_3 dx dt \Big| & \lesssim \|v_1\|_{L^3_x L^{2}_t} 
\|v_2\|_{L^6_x L^{2+}_t} \|v_3\|_{L^2_x L^{\infty-}_t} \\
&\lesssim \|v_1\|_{X^{\frac{1}{12},\frac{1}{4}+}_{|\tau|=|\xi|}} \|v_2\|_{X^{\frac{1}{6}+,\frac{1}{2}+}_{|\tau|=|\xi|}    } \|v_3\|_{X^{0,\frac{1}{2}-}_{\tau =0}} \, , \\
\end{align*}
which gives the desired estimate.\\
Case 2: The first term in (\ref{angle}) is dominant (and $|\xi_1| \le |\xi_2|$). \\
We have to show
\begin{align*}
\Big| \int_* \frac{\widehat{u}_1(\tau_1,\xi_1)}{\langle |\tau_1| - |\xi_1|\rangle^{\epsilon-} \langle \xi_1 \rangle^{\frac{1}{2}}} 
\frac{\widehat{u}_2(\tau_2,\xi_2)}{\langle |\tau_2| - |\xi_2|\rangle^{\frac{1}{2}+\epsilon-}}
\frac{\widehat{u}_3(\tau_3,\xi_3)}{\langle \xi_3 \rangle^{\frac{1}{2}-\epsilon} \langle \tau_3 \rangle^{\frac{1}{2}-}} \Big|
\lesssim \prod_{i=1}^3 \|u_i\|_{L^2_{xt}} \, .
\end{align*}
2.1: $|\tau_2| \ll|\xi_2|$ $\Rightarrow$ $\langle |\tau_2|-|\xi_2| \rangle \sim \langle \xi_2 \rangle$. \\
The last estimate is reduced to
\begin{align*}
\Big| \int_* \frac{\widehat{u}_1(\tau_1,\xi_1)}{\langle |\tau_1| - |\xi_1|\rangle^{\epsilon-} \langle \xi_1 \rangle^{\frac{1}{2}}} 
\frac{\widehat{u}_2(\tau_2,\xi_2)}{\langle \tau_2 \rangle^{0+} \langle \xi_2 \rangle^{\frac{1}{2}+\epsilon--}}
\frac{\widehat{u}_3(\tau_3,\xi_3)}{\langle \xi_3 \rangle^{\frac{1}{2}-\epsilon} \langle \tau_3 \rangle^{\frac{1}{2}-}} \Big|
\lesssim \prod_{i=1}^3 \|u_i\|_{L^2_{xt}} \, ,
\end{align*}
which follows by Sobolev's embedding from the estimate
\begin{align*}
\Big| \int v_1 v_2 v_3 dx dt \Big| & \lesssim \|v_1\|_{L^4_x L^{2}_t} 
\|v_2\|_{L^4_x L^{2+}_t} \|v_3\|_{L^2_x L^{\infty-}_t} \\
&\lesssim \|v_1\|_{X^{\frac{1}{2},0}_{|\tau|=|\xi|}} \|v_2\|_{X^{\frac{1}{2},0+}_{\tau=0}} \|v_3\|_{X^{0,\frac{1}{2}-}_{\tau =0}} \, . \\
\end{align*}
2.2: $|\tau_2| \ge |\xi_2|$. \\
2.2.1: $|\tau_1| \lesssim |\xi_1|$.
We use
$$ 1 \lesssim \frac{\langle \tau_2 \rangle^{\frac{1}{2}-}}{\langle \xi_2 \rangle^{\frac{1}{2}-}} \lesssim \frac{\langle \tau_1 \rangle^{\frac{1}{2}-}}{\langle \xi_2 \rangle^{\frac{1}{2}-}} + \frac{\langle \tau_3 \rangle^{\frac{1}{2}-}}{\langle \xi_2 \rangle^{\frac{1}{2}-}} \, . $$
If the first term is dominant we reduce to
\begin{align*}
\Big| \int_* \frac{\widehat{u}_1(\tau_1,\xi_1)}{\langle |\tau_1| - |\xi_1|\rangle^{\epsilon-} \langle \xi_1 \rangle^{0+} }
\frac{\widehat{u}_2(\tau_2,\xi_2)}{\langle |\tau_2|-|\xi_2| \rangle^{\frac{1}{2}+\epsilon-} \langle \xi_2 \rangle^{\frac{1}{2}-}}
\frac{\widehat{u}_3(\tau_3,\xi_3)}{\langle \xi_3 \rangle^{\frac{1}{2}-\epsilon} \langle \tau_3 \rangle^{\frac{1}{2}-}} \Big|
\lesssim \prod_{i=1}^3 \|u_i\|_{L^2_{xt}} \, .
\end{align*}
It is a consequence of the following estimate which follows from (\ref{I1}) and Sobolev:
\begin{align*}
\Big| \int v_1 v_2 v_3 dx dt \Big| & \lesssim \|v_1\|_{L^2_x L^{2}_t} 
\|v_2\|_{L^6_x L^{2+}_t} \|v_3\|_{L^3_x L^{\infty-}_t} \\
&\lesssim \|v_1\|_{X^{0,0}_{|\tau|=|\xi|}} \|v_2\|_{X^{\frac{1}{6}+,\frac{1}{2}+}_{|\tau|=|\xi|}} \|v_3\|_{X^{\frac{1}{3},\frac{1}{2}-}_{\tau =0}} \, . \\
\end{align*}
If the second term is dominant we need
\begin{align*}
\Big| \int_* \frac{\widehat{u}_1(\tau_1,\xi_1)}{\langle |\tau_1| - |\xi_1|\rangle^{\epsilon-} \langle \xi_1 \rangle^{\frac{1}{2}}} 
\frac{\widehat{u}_2(\tau_2,\xi_2)}{\langle |\tau_2|-|\xi_2| \rangle^{\frac{1}{2}+\epsilon-} \langle \xi_2 \rangle^{\frac{1}{2}-}}
\frac{\widehat{u}_3(\tau_3,\xi_3)}{\langle \xi_3 \rangle^{\frac{1}{2}-\epsilon}} \Big|
\lesssim \prod_{i=1}^3 \|u_i\|_{L^2_{xt}} \, ,
\end{align*}
which follows from Theorem \ref{Theorem3}.\\
2.2.2: $|\tau_1| \gg |\xi_1|$ $\Rightarrow$ $\langle |\tau_1|-|\xi_1|\rangle \sim \langle \tau_1 \rangle$. \\
We use
$$ 1 \lesssim \frac{\langle \tau_2 \rangle^{\frac{\epsilon}{4}}}{\langle \xi_2 \rangle^{\frac{\epsilon}{4}}} \lesssim \frac{\langle \tau_1 \rangle^{\frac{\epsilon}{4}}}{\langle \xi_2 \rangle^{\frac{\epsilon}{4}}} + \frac{\langle \tau_3 \rangle^{\frac{\epsilon}{4}}}{\langle \xi_2 \rangle^{\frac{\epsilon}{4}}} \, . $$
The first term reduces to
\begin{align*}
\Big| \int_* \frac{\widehat{u}_1(\tau_1,\xi_1)}{ \langle\tau_1\rangle^{\frac{3}{4}\epsilon-} \langle \xi_1 \rangle^{\frac{1}{2}}} 
\frac{\widehat{u}_2(\tau_2,\xi_2)}{\langle |\tau_2|-|\xi_2| \rangle^{\frac{1}{2}+\epsilon-} \langle \xi_2 \rangle^{\frac{\epsilon}{4}}}
\frac{\widehat{u}_3(\tau_3,\xi_3)}{\langle \tau_3 \rangle^{\frac{1}{2}-} \langle \xi_3 \rangle^{\frac{1}{2}-\epsilon}} \Big|
\lesssim \prod_{i=1}^3 \|u_i\|_{L^2_{xt}} \, .
\end{align*}
We estimate using (\ref{I6}) and Sobolev:
\begin{align*}
\Big| \int v_1 v_2 v_3 dx dt \Big| & \lesssim \|v_1\|_{L^4_x L^{2+}_t} 
\|v_2\|_{L^{\frac{2}{1-\epsilon}}_x L^{2}_t} \|v_3\|_{L^{\frac{4}{1+2\epsilon}}_x L^{\infty-}_t} \\
& \lesssim \|v_1\|_{H^{\frac{1}{2}}_x H^{0+}_t} 
\|v_2\|_{X^{\frac{\epsilon}{4},\frac{3}{4}\epsilon+}_{|\tau|=|\xi|}} \|v_3\|_{H^{\frac{1}{2}-\epsilon}_x H^{\frac{1}{2}-}_t} \\
&\lesssim \|v_1\|_{X^{\frac{1}{2},0+}_{\tau=0}} \|v_2\|_{X^{\frac{\epsilon}{4},\frac{3}{4}\epsilon-}_{|\tau|=|\xi|}} \|v_3\|_{X^{\frac{1}{2}-\epsilon,\frac{1}{2}-}_{\tau =0}} \, , \\
\end{align*}
which gives the desired bound.\\
If the second term is dominant we have to show
\begin{align*}
\Big| \int_* \frac{\widehat{u}_1(\tau_1,\xi_1)}{ \langle \tau_1 \rangle^{\epsilon-} \langle \xi_1 \rangle^{\frac{1}{2}}} 
\frac{\widehat{u}_2(\tau_2,\xi_2)}{\langle |\tau_2|-|\xi_2| \rangle^{\frac{1}{2}+\epsilon-} \langle \xi_2 \rangle^{\frac{\epsilon}{4}}}
\frac{\widehat{u}_3(\tau_3,\xi_3)}{\langle \tau_3 \rangle^{\frac{1}{2}-\frac{\epsilon}{4}-} \langle \xi_3 \rangle^{\frac{1}{2}-\epsilon}} \Big|
\lesssim \prod_{i=1}^3 \|u_i\|_{L^2_{xt}} \, ,
\end{align*}
which follows from
\begin{align*}
\Big| \int v_1 v_2 v_3 dx dt \Big| & \lesssim \|v_1\|_{L^4_x L^{\frac{4}{2-\epsilon}+}_t} 
\|v_2\|_{L^{\frac{2}{1-\epsilon}}_x L^{2}_t} \|v_3\|_{L^{\frac{4}{1+2\epsilon}}_x L^{\frac{4}{\epsilon}-}_t} \\
& \lesssim \|v_1\|_{H^{\frac{1}{2}}_x H^{\frac{\epsilon}{4}+}_t} 
\|v_2\|_{H^{\frac{\epsilon}{4}}_x L^2_t} \|v_3\|_{H^{\frac{1}{2}-\epsilon}_x H^{\frac{1}{2}-\frac{\epsilon}{4}-}_t} \\
&\lesssim \|v_1\|_{X^{\frac{1}{2},\frac{\epsilon}{4}+}_{\tau=0}} \|v_2\|_{X^{\frac{\epsilon}{4},0}_{|\tau|=|\xi|}} \|v_3\|_{X^{\frac{1}{2}-\epsilon,\frac{1}{2}-\frac{\epsilon}{4}-}_{\tau =0}} \, , \\
\end{align*}
where we used Sobolev's embedding.\\
Case 3: The second term in (\ref{angle}) is dominant (and $|\xi_1| \le \xi_2|$). \\
In this case we reduce to
\begin{align*}
\Big| \int_* \frac{\widehat{u}_1(\tau_1,\xi_1)}{ \langle |\tau_1|-|\xi_1| \rangle^{\frac{1}{2}+\epsilon-} \langle \xi_1 \rangle^{\frac{1}{2}}} 
\frac{\widehat{u}_2(\tau_2,\xi_2)}{\langle |\tau_2|-|\xi_2| \rangle^{\epsilon-}}
\frac{\widehat{u}_3(\tau_3,\xi_3)}{\langle \tau_3 \rangle^{\frac{1}{2}-} \langle \xi_3 \rangle^{\frac{1}{2}-\epsilon}} \Big|
\lesssim \prod_{i=1}^3 \|u_i\|_{L^2_{xt}} \, ,
\end{align*}
which follows by Strichartz' (\ref{Str}) and Sobolev's estimates by
\begin{align*}
\Big| \int v_1 v_2 v_3 dx dt \Big| & \lesssim \|v_1\|_{L^6_x L^6_t} 
\|v_2\|_{L^2_x L^{2}_t} \|v_3\|_{L^3_x L^3_t} \\
&\lesssim \|v_1\|_{X^{\frac{1}{2},\frac{1}{2}+}_{|\tau|=|\xi|}} \|v_2\|_{X^{0,0}_{|\tau|=|\xi|}} \|v_3\|_{X^{\frac{1}{3},\frac{1}{6}}_{\tau =0}} \, . \\
\end{align*}
Claim 1 is now proven.

Before proceeding we estimate $A^{df}$. By Sobolev's embedding $\dot{H}^{1-\epsilon,\frac{2}{2-\epsilon}} \subset L^2$ we obtain:
\begin{align}
\label{23}
\||\nabla|^{\epsilon}A_i^{df}\|_{L^{\infty}_t L^2_x} & \lesssim \| \phi \partial_t \phi \|_{L^{\infty}_t \dot{H}^{-1+\epsilon}_x} \lesssim \| \phi \partial_t \phi \|_{L^{\infty}_t L^{\frac{2}{2-\epsilon}}_x} \\
\nonumber
& \lesssim \|\phi\|_{L^{\infty}_t L^{\frac{2}{1-\epsilon}}_x} \| \partial_t \phi \|_{L^{\infty}_t L^2_x} \lesssim \|\phi\|_{X^{1,\frac{1}{2}+}_{|\tau|=|\xi|}} \|\partial_t \phi\|_{X^{0,\frac{1}{2}+}_{|\tau|=|\xi|}} 
\end{align}
and for $\epsilon >0$:
\begin{equation}
\label{24}
\| \nabla A_i^{df}\|_{X^{0,\frac{1}{2}-}_{|\tau|=|\xi|}} \lesssim \|\phi \partial_t \phi \|_{X^{0,\frac{1}{2}-}_{|\tau|=|\xi|}} \lesssim \|\phi\|_{X^{1,\frac{1}{2}+\epsilon-}_{|\tau|=|\xi|}} \|\partial_t \phi\|_{X^{0,\frac{1}{2}+\epsilon-}_{|\tau|=|\xi|}}
\end{equation} 
by Theorem \ref{Theorem3}, which for $0 < \delta < \epsilon$ also implies
\begin{equation}
\label{25}
\| \nabla A_i^{df}\|_{X^{-\delta,\frac{1}{2}+\delta-}_{|\tau|=|\xi|}} \lesssim \|\phi \partial_t \phi \|_{X^{-\delta,\frac{1}{2}+\delta-}_{|\tau|=|\xi|}} \lesssim \|\phi\|_{X^{1,\frac{1}{2}+\epsilon-}_{|\tau|=|\xi|}} \|\partial_t \phi\|_{X^{0,\frac{1}{2}+\epsilon-}_{|\tau|=|\xi|}} \, .
\end{equation} 

The cubic terms are easier to handle, because they contain one derivative less.\\
{\bf Claim 2:}
 \begin{align*}
&\|A_i \partial_j (|\phi|^2)\|_{L^2_t H^{-\frac{1}{2}+\epsilon}_x} \\ 
& \lesssim (\| \nabla A_i^{cf}\|_{X^{-\frac{1}{2}+\epsilon,\frac{1}{2}+}_{|\tau|=|\xi|}} +  \||\nabla|^{\epsilon}A_i^{cf}\|_{L^{\infty}_t(L^2_x)}) \|\phi \|^2_{X^{1,\frac{1}{2}+}_{|\tau|=|\xi|}} 
 +\|\phi\|_{X^{1,\frac{1}{2}+}_{|\tau|=|\xi|}}^3  \|\partial_t \phi\|_{X^{0,\frac{1}{2}+}_{|\tau|=|\xi|}}  \, .
\end{align*}
We split $A=A^{cf} + A^{df}$, and moreover $A^{cf}_i = A^{cf,h}_i + A^{cf,l}_i$ as well as $A^{df}_i = A^{df,h}_i + A^{df,l}_i$ into their low and high frquency parts, i.e. $supp \,  \widehat{A}_i^{cf,h} \subset \{|\xi| \ge 1 \}$ , $supp \, \widehat{A}_i^{cf,l} \subset \{|\xi| \le 1 \}$ and similarly $A^{df}_i$.\\
For the high frequency parts we obtain by (\ref{Str}) and Sobolev
\begin{align*}
&\|A_i^{cf,h} \partial_j (|\phi|^2)\|_{L^2_t H^{-\frac{1}{2}+\epsilon}_x} 
 \lesssim \|A_i^{cf,h} \partial_j (|\phi|^2)\|_{L^2_t L^{\frac{4}{3}+\epsilon'}_x} \\
& \lesssim \| A^{cf,h}_i\|_{L^6_t L^6_x} \|\phi\|_{L^3_t L^{12+\epsilon''}_x} \|\partial_j \phi\|_{L^{\infty}_t L^2_x} 
 \lesssim \| A_i^{cf,h}\|_{X^{\frac{1}{2},\frac{1}{2}+}_{|\tau|=|\xi|}}  \|\phi\|^2_{X^{1,\frac{1}{2}+}_{|\tau|=|\xi|}}  \\
& \lesssim \| \nabla A_i^{cf}\|_{X^{-\frac{1}{2}+\epsilon,\frac{1}{2}+}_{|\tau|=|\xi|}}  \|\phi\|^2_{X^{1,\frac{1}{2}+}_{|\tau|=|\xi|}} 
\end{align*}
and using also (\ref{24}) we obtain
\begin{align*}
&\|A_i^{df,h} \partial_j (|\phi|^2)\|_{L^2_t H^{-\frac{1}{2}+\epsilon}_x} 
  \lesssim \| A^{df,h}_i\|_{L^6_t L^6_x} \|\phi\|^2_{X^{1,\frac{1}{2}+}_{|\tau|=|\xi|}}  \\
& \lesssim \| \nabla A_i^{df,h}\|_{X^{0,\frac{1}{2}-}_{|\tau|=|\xi|}}  \|\phi\|^2_{X^{1,\frac{1}{2}+}_{|\tau|=|\xi|}} 
 \lesssim \|\phi\|^3_{X^{1,\frac{1}{2}+\epsilon-}_{|\tau|=|\xi|}}  \|\partial_t \phi\|_{X^{0,\frac{1}{2}+\epsilon-}_{|\tau|=|\xi|}} \, .
\end{align*}
The low frequency part is taken care of as follows
\begin{align*}
&\|A_i^{cf,l} \partial_j (|\phi|^2)\|_{L^2_t H^{-\frac{1}{2}+\epsilon}_x}
\lesssim \|A_i^{cf,l} \partial_j (|\phi|^2)\|_{L^2_t L^{\frac{4}{3}+\epsilon'}_x} \\
&\lesssim \|A_i^{cf,l}\|_{L^{\infty}_t L^8_x} \|\phi\|_{L^4_t L^{8+\epsilon''}_x} \|\nabla \phi\|_{L^4_t L^2_x} 
 \lesssim \||\nabla|^{\epsilon}A_i^{cf}\|_{L^{\infty}_t L^2_x} \|\phi\|^2_{X^{1,\frac{1}{2}+}_{|\tau|=|\xi|}} \, .
 \end{align*}
 Similarly we obtain by (\ref{23}):
 \begin{align*}
 &\|A_i^{df,l} \partial_j (|\phi|^2)\|_{L^2_t H^{-\frac{1}{2}+\epsilon}_x}
\lesssim \||\nabla|^{\epsilon}A_i^{df}\|_{L^{\infty}_t L^2_x} \|\phi\|^2_{X^{1,\frac{1}{2}+}_{|\tau|=|\xi|}}
 \lesssim \|\phi\|^3_{X^{1,\frac{1}{2}+}_{|\tau|=|\xi|}} \|\partial_t \phi\|_{X^{0,\frac{1}{2}+}_{|\tau|=|\xi|}} \, .
 \end{align*}
{\bf Claim 3:} 
 $$\|Im{(\overline{\phi} \partial_t \phi) |\phi|^2 \|_{L^2_t H^{-\frac{1}{2}+\epsilon}_x}     \lesssim \|\phi\|^3_{X^{1,\frac{1}{2}+}_{|\tau|=|\xi|}} \|\partial_t \phi\|_{X^{0,\frac{1}{2}+}_{|\tau|=|\xi|}}} \, . $$
This follows from
\begin{align*}
\|Im(\overline{\phi} \partial_t \phi) |\phi|^2\|_{L^2_t H^{-\frac{1}{2}+\epsilon}_x} 
 \lesssim \|Im(\overline{\phi} \partial_t \phi) |\phi|^2\|_{L^2_t L^{\frac{4}{3}+\epsilon'}_x}
\lesssim \|\phi\|^3_{L^{\infty}_t L^{12+\epsilon''}_x}  \|\partial_t \phi\|_{L^2_t L^2_x}
\, .
\end{align*}

In order to control $\|A^{cf}\|_{L^{\infty}_t([0,T],L^2_x)}$ in the fixed point argument we only have to consider the low frequency part of $A^{cf}$, because the high frequency part is controlled by $\| \nabla A^{cf,h}\|_{X^{-\frac{1}{2}+\epsilon,\frac{1}{2}+}_{|\tau|=|\xi|}}$. Denote the projection onto the low frequency part of $A^{cf}_i$ by $P$. \\
{\bf Claim 4 (a)} 
$$ \int_0^T \|P(\nabla \phi \nabla \phi)\|_{\dot{H}^{-1+\epsilon}_x} dt \lesssim \int_0^T \|\nabla \phi\|_{L^2_x}^2 dt  \lesssim T \|\nabla \phi\|^2_{X^{0,\frac{1}{2}+}_{|\tau|=|\xi|}} $$
by Sobolev's embedding $\dot{H}_x^{1+\epsilon} \cap \dot{H}_x^{1-\epsilon} \subset L^{\infty}_x$, which implies
\begin{align*}
&\Big| \int P(\nabla \phi \nabla \phi)w dx \Big| = \Big| \int \nabla \phi \nabla \phi Pw dx \Big| \lesssim \|\nabla \phi\|_{L^2}^2 \|Pw\|_{L^{\infty}} \\
& \lesssim \|\nabla \phi\|_{L^2}^2(\||\nabla|^{1-\epsilon} Pw\|_{L^2} + \||\nabla|^{1+\epsilon}Pw\|_{L^2}) \lesssim \|\nabla \phi\|_{L^2}^2 \||\nabla|^{1-\epsilon} w\|_{L^2} \, .
\end{align*}
{\bf(b)} 
\begin{align*}
&\int_0^T \| A^{df} \nabla(|\phi|^2)\|_{\dot{H}^{-1+\epsilon}_x} dt  \lesssim \int_0^T \| A^{df} \nabla(|\phi|^2)\|_{L^{\frac{2}{2-\epsilon}}_x} dt \\
 &\lesssim \int_0^T \|A^{df} \|_{L^{\frac{2}{1-\epsilon}+}_x} \| \phi\|_{L^{\infty-}_x}  \|\nabla \phi\|_{L^2_x} dt 
 \lesssim T \||\nabla|^{\epsilon+} A^{df} \|_{L^{\infty}_t L^2_x} \| \phi\|_{L^{\infty}_t H^1_x}  \|\nabla \phi\|_{L^{\infty}_t L^2_x} \\
 & \hspace{16em}\lesssim T \|\phi\|^3_{X^{1,\frac{1}{2}+}_{|\tau|=|\xi|}}  \|\partial_t \phi\|_{X^{0,\frac{1}{2}+}_{|\tau|=|\xi|}}
\end{align*}
by (\ref{23}).\\
{\bf (c)} 
\begin{align*}
 \int_0^T \| A^{cf} \nabla(|\phi|^2)\|_{\dot{H}^{-1+\epsilon}_x} dt  &
\lesssim T \|A^{cf}\|_{L^{\infty}_t L^{\frac{2}{1-\epsilon}+}_x} \|\phi\|_{L^{\infty}_t L^{\infty-}_x} \|\nabla \phi\|_{L^{\infty}_t L^2_x} \\
& \lesssim T \||\nabla|^{\epsilon+}A^{cf}\|_{L^{\infty}_t L^2_x}\|\phi\|^2_{X^{1,\frac{1}{2}+}_{|\tau|=|\xi|}} \\
&\lesssim T (\||\nabla|^{\epsilon} A^{cf}\|_{L^{\infty}_t L^2_x} + \| \nabla A^{cf}\|_{X^{-\frac{1}{2}+\epsilon,\frac{1}{2}+}_{\tau=0}})\|\phi\|^2_{X^{1,\frac{1}{2}+}_{|\tau|=|\xi|}}
\end{align*}
{\bf (d)} 
\begin{align*}
&\int_0^T \|\phi (\partial_t \phi) |\phi|^2\|_{\dot{H}^{-1+\epsilon}_x} dt \lesssim \int_0^T \|\phi (\partial_t \phi) |\phi|^2\|_{L^\frac{2}{2-\epsilon}} dt \\
& \lesssim T \|\phi\|^3_{L^{\infty}_x H^1_x} \|\partial_t \phi\|_{L^{\infty}_t L^2_x} \lesssim T \|\phi\|^3_{X^{1,\frac{1}{2}+}_{|\tau|=|\xi|}} \|\partial_t \phi\|_{X^{0,\frac{1}{2}+}_{|\tau|=|\xi|}} \, .
\end{align*}
Next in order to estimate $\|\phi\|_{X^{1,\frac {1}{2}+\epsilon-}_{|\tau|=|\xi|}}$ and $\|\partial_t \phi\|_{X^{0,\frac {1}{2}+\epsilon-}_{|\tau|=|\xi|}}$ we have to control the right hand side of (\ref{***1'}).\\
{\bf Claim 5:}
$$ \|A^{df} \nabla \phi\|_{X^{0,-\frac{1}{2}+\epsilon}_{|\tau|=|\xi|}} 
\lesssim (\||\nabla|^{\epsilon} A^{df}\|_{L^{\infty}_t L^2_x} + \|\nabla A^{df}\|_{X^{0,\frac{1}{2}-}_{|\tau|=|\xi|}}) \|\nabla \phi\|_{X^{0,\frac{1}{2}+\epsilon-}_{|\tau|=|\xi|}}
\lesssim
\|\phi\|^3_{X^{1,\frac{1}{2}+\epsilon-}_{|\tau|=|\xi|}} \, . $$
The last estimate follows from (\ref{23}) and (\ref{24}). For the first estimate we consider the low and high frequency parts of $A^{df}$ as follows:
\begin{align*}
&\|A^{df,l} \nabla \phi \|_{L^2_{xt}} \lesssim T^{\frac{1}{2}} \|A^{df,l}\|_{L^{\infty}_t L^{\infty}_x} \|\nabla \phi \|_{L^{\infty}_t L^{2}_x} \\
& \lesssim T^{\frac{1}{2}} \||\nabla|^{1-}A^{df,l}\|_{L^{\infty}_t L^2_x} \|\nabla \phi\|_{L^{\infty}_t L^2_x}
 \lesssim T^{\frac{1}{2}} \||\nabla|^{\epsilon}A^{df}\|_{L^{\infty}_t L^2_x} \|\nabla \phi\|_{X^{0,\frac{1}{2}+}_{|\tau|=|\xi|}} \, ,\\
& \|A^{df,h} \nabla \phi\|_{X^{0,-\frac{1}{2}+\epsilon}_{|\tau|=|\xi|}} 
\lesssim \| A^{df,h}\|_{X^{1,\frac{1}{2}-}_{|\tau|=|\xi|}}) \|\nabla \phi\|_{X^{0,\frac{1}{2}+\epsilon-}_{|\tau|=|\xi|}} \, ,
\end{align*}
where the last estimate follows by Theorem \ref{Theorem3}. \\
{\bf Claim 6:}
$$ \|A^{cf} \nabla \phi\|_{L^2_{xt}} \lesssim (\|\nabla A^{cf}\|_{X^{-\frac{1}{2}+,\frac{1}{2}+}_{\tau =0}} + \||\nabla|^{\epsilon} A^{cf}\|_{L^{\infty}_t L^2_x}) \|\nabla \phi\|_{X^{0,\frac{1}{2}+\epsilon-}_{|\tau|=|\xi|}} \, , $$
where we split $A^{cf}$ into its low and high frequency parts $A^{cf}_l$ and $A^{cf}_h$. The low frequency part is easily estimated as follows:
$$ \|A^{cf,l}\nabla \phi\|_{L^2_{xt}} \lesssim \|A^{cf,l} \|_{L^{\infty}_t L^{\infty}_x} \|\nabla \phi\|_{L^2_t L^2_x} \lesssim \||\nabla|^{\epsilon}A^{cf,l}\|_{L^{\infty}_t L^2_x} \|\nabla \phi\|_{X^{0,\frac{1}{2}+\epsilon-}_{|\tau|=|\xi|}} \, . $$
For the high frequency part we want to show:
$$\|A^{cf,h} \nabla \phi \|_{L^2_{xt}} \lesssim \|A^{cf,h}\|_{X^{\frac{1}{2}+,\frac{1}{2}+}_{\tau=0}} \|\nabla \phi \|_{X^{0,\frac{1}{2}+}_{|\tau|=|\xi|}} \, . $$
This estimate would follow if we prove
$$
\int_* m(\xi_1,\xi_2,\xi_3,\tau_1,\tau_2,\tau_3) \widehat{u}_1(\xi_1,\tau_1)  \widehat{u}_2(\xi_2,\tau_2) \widehat{u}_3(\xi_3,\tau_3) d\xi d\tau \lesssim \prod_{i=1}^3 \|u_i\|_{L^2_{xt}} \, , 
$$
where 
$$ m = \frac{1}{ \langle |\tau_2| - |\xi_2|\rangle^{\frac{1}{2}+}  \langle \xi_3 \rangle^{\frac{1}{2}+}\langle \tau_3 \rangle^{\frac{1}{2}+}} \, .$$
The following argument is closely related to the proof of a similar estimate in  \cite{T1}.\\
By two applications of the averaging principle (\cite{T}, Prop. 5.1) we may replace $m$ by
$$ m' = \frac{ \chi_{||\tau_2|-|\xi_2||\sim 1} \chi_{|\tau_3| \sim 1}}{ \langle \xi_3 \rangle^{\frac{1}{2}+}} \, . $$
Let now $\tau_2$ be restricted to the region $\tau_2 =T + O(1)$ for some integer $T$. Then $\tau_1$ is restricted to $\tau_1 = -T + O(1)$, because $\tau_1 + \tau_2 + \tau_3 =0$, and $\xi_2$ is restricted to $|\xi_2| = |T| + O(1)$. The $\tau_1$-regions are essentially disjoint for $T \in {\mathbb Z}$ and similarly the $\tau_2$-regions. Thus by Schur's test (\cite{T}, Lemma 3.11) we only have to show
\begin{align*}
 &\sup_{T \in {\mathbb Z}} \int_* \frac{ \chi_{\tau_1=-T+O(1)} \chi_{\tau_2=T+O(1)} \chi_{|\tau_3|\sim 1} \chi_{|\xi_2|=|T|+O(1)}}{\langle \xi_3 \rangle^{\frac{1}{2}+}}\cdot \\
 & \hspace{14em} \cdot\widehat{u}_1(\xi_1,\tau_1) \widehat{u}_2(\xi_2,\tau_2)
\widehat{u}_3(\xi_3,\tau_3) d\xi d\tau \lesssim \prod_{i=1}^3 \|u_i\|_{L^2_{xt}} \, . 
\end{align*}
The $\tau$-behaviour of the integral is now trivial, thus we reduce to
\begin{equation}
\label{50}
\sup_{T \in {\mathbb N}} \int_{\sum_{i=1}^3 \xi_i =0}  \frac{ \chi_{|\xi_2|=|T|+O(1)}}{ \langle \xi_3 \rangle^{\frac{1}{2}+}} \widehat{f}_1(\xi_1)\widehat{f}_2(\xi_2)\widehat{f}_2(\xi_3)d\xi \lesssim \prod_{i=1}^3 \|f_i\|_{L^2_x} \, .
\end{equation}
It only remains to consider the following two cases: \\
Case 1: $|\xi_1| \sim |\xi_3| \gtrsim T$. We obtain in this case
\begin{align*}
L.H.S. \, of \, (\ref{50}) 
&\lesssim \sup_{T \in{\mathbb N}} \frac{1}{T^{\frac{1}{2}+}} \|f_1\|_{L^2} \|f_3\|_{L^2} \| {\mathcal F}^{-1}(\chi_{|\xi|=T+O(1)} \widehat{f}_2)\|_{L^{\infty}({\mathbb R}^2)} \\
&\lesssim \sup_{T \in{\mathbb N}} \frac{1}{ T^{\frac{1}{2}+}} 
\|f_1\|_{L^2} \|f_3\|_{L^2} \| \chi_{|\xi|=T+O(1)} \widehat{f}_2\|_{L^1({\mathbb R}^2)} \\
&\lesssim \hspace{-0.1em}\sup_{T \in {\mathbb N}} \frac{T^{\frac{1}{2}}}{T^{\frac{1}{2}+}}  \prod_{i=1}^3 \|f_i\|_{L^2} \lesssim\hspace{-0.1em}
\prod_{i=1}^3 \|f_i\|_{L^2}.
\end{align*}
Case 2: $|\xi_1| \sim T \gtrsim |\xi_3|$. 
An elementary calculation shows that
\begin{align*}
L.H.S. \, of \,  (\ref{50})
\lesssim \sup_{T \in{\mathbb N}} \| \chi_{|\xi|=T+O(1)} \ast \langle \xi \rangle^{-1-}\|^{\frac{1}{2}}_{L^{\infty}(\mathbb{R}^2)} \prod_{i=1}^3 \|f_i\|_{L^2_x} \lesssim \prod_{i=1}^3 \|f_i\|_{L^2_x}\, ,
\end{align*}
so that the desired estimate follows.\\
{\bf Claim 7:}
$$ \|\nabla A^{cf} \phi\|_{X^{0,-\frac{1}{2}+\epsilon-}_{|\tau|=|\xi|}} \lesssim \| \nabla A^{cf}\|_{X^{-\frac{1}{2}+\epsilon,\frac{1}{2}+}_{\tau =0}} \|\phi\|_{X^{1,\frac{1}{2}+\epsilon--}_{|\tau|=|\xi|}} \, . $$
By duality this is equivalent to
$$ \| w \phi \|_{X^{\frac{1}{2}-\epsilon,-\frac{1}{2}-}_{\tau =0}} \lesssim \|w\|_{X^{0,\frac{1}{2}-\epsilon+}_{|\tau|=|\xi|}} \|\phi\|_{X^{1,\frac{1}{2}+\epsilon--}_{|\tau|=|\xi|}} \, . $$
We use the estimate $ \frac{\langle \xi \rangle}{\langle \tau \rangle}  \lesssim \langle |\xi| - |\tau| \rangle$ and obtain
$$ \| w \phi \|_{X^{\frac{1}{2}-\epsilon,-\frac{1}{2}-}_{\tau =0}} \lesssim \|w \phi\|_{X^{0,\frac{1}{2}-\epsilon}_{|\tau|=|\xi|}} \lesssim \|w\|_{X^{0,\frac{1}{2}-\epsilon+}_{|\tau|=|\xi|}} \|\phi\|_{X^{1,\frac{1}{2}+\epsilon--}_{|\tau|=|\xi|}} \, , $$
where the last estimate follows from Theorem \ref{Theorem3} with $s_0 = 0$ , $b_0 = -\frac{1}{2}+\epsilon$ , $s_1 = 0$ , $s_2=1 $ , $b_1=\frac{1}{2}-\epsilon+$ , $b_2=\frac{1}{2}+\epsilon--$ . \\
{\bf Claim 8:} 
$$\|A^{cf} A^{cf} \phi \|_{L^2_{xt}} \lesssim (\|\nabla A^{cf}\|^2_{X^{-\frac{1}{2}+\epsilon,\frac{1}{2}+}_{\tau =0}} + \||\nabla|^{\epsilon}A^{cf}\|^2_{L^{\infty}_t L^2_x}) \|\phi\|_{X^{1,\frac{1}{2}+}_{|\tau|=|\xi|}} \, . $$
Splitting $A^{cf}= A^{cf,h} + A^{cf,l}$ we first consider 
\begin{align*}
\|A^{cf,h} A^{cf,h} \phi\|_{L^2_{xt}} 
\lesssim  \| A^{cf,h}\|_{L^4_t L^{4+}_x}^2  \phi\|_{L^{\infty}_t L^{\infty-}_x}  \lesssim \|A^{cf,h}\|_{X^{\frac{1}{2}+\epsilon,\frac{1}{4}}_{\tau =0}}^2  \|\phi\|_{X^{1,\frac{1}{2}+}_{|\tau|=|\xi|}} \, .
\end{align*}
Next we consider
$$
\|A^{cf,l} A^{cf,l} \phi\|_{L^2_{xt}}  \lesssim \|A^{cf,l}\|_{L^{\infty}_t L^{\infty}_x}^2 \|\phi\|_{L^2_t L^2_x} 
\lesssim \||\nabla|^{\epsilon}A^{cf}\|_{L^{\infty}_t L^2_x}^2  \|\phi\|_{X^{1,\frac{1}{2}+}_{|\tau|=|\xi|}}  $$
and also
\begin{align*}
\|A^{cf,l} A^{cf,h} \phi\|_{L^2_{xt}} & \lesssim \|A^{cf,l}\|_{L^{\infty}_t L^{\infty}_x} \|A^{cf,h}\|_{L^4_t L^{2+}_x} \|\phi\|_{L^4_t L^{\infty-}_x} \\ &\lesssim \||\nabla|^{\epsilon}A^{cf,l}\|_{L^{\infty}_t L^2_x} \| A^{cf,h} \|_{X^{\frac{1}{2}+,\frac{1}{4}}_{\tau|=0}} \|\phi\|_{X^{1,\frac{1}{2}+}_{|\tau|=|\xi|}} \, ,
\end{align*}
which completes the proof of claim 8.

If one combines similar estimates with (\ref{23}) and (\ref{24}) we also obtain the required bounds for $\|A^{df} A^{df} \phi\|_{L^2_{xt}}$ and $\|A^{df} A^{cf} \phi\|_{L^2_{xt}}$.\\
{\bf Claim 9:} For a suitable $N \in {\mathbb N}$ the following estimate holds:
$$\|\phi V'(|\phi|^2)\|_{L^2_{xt}} \lesssim \|\phi\|_{X^{1,\frac{1}{2}+}_{|\tau|=|\xi|}} (1 + \|\phi\|_{X^{1,\frac{1}{2}+}_{|\tau|=|\xi|}}^N) \, .$$
Using the polynomial bound of  $V'$ we obtain:
$$ \|\phi V'(|\phi|^2)\|_{L^2_{xt}} \lesssim \|\phi\|_{L^2_{xt}} +\|\phi\|^{N+1}_{L^{2(N+1)}_{xt}} \lesssim \|\phi\|_{X^{1,\frac{1}{2}+}_{|\tau|=|\xi|}} (1 + \|\phi\|_{X^{1,\frac{1}{2}+}_{|\tau|=|\xi|}}^N) \, .$$

Now the contraction mapping principle applies. The claimed properties of $A^{df}$ follow immediately from (\ref{23}), (\ref{24}) and (\ref{25}). The proof of Theorem \ref{Theorem1.1} is complete.
\end{proof}

\section{Proof of Theorem \ref{Theorem1.2}}
\begin{proof}
We follow the arguments of Selberg-Tesfahun \cite{ST} in the case of the Lorenz gauge. Define
$$ I(t) = \|\phi(t)\|_{L^2} + \sum_{j=1}^2 \|D_j \phi(t)\|_{L^2} + \|\partial_t \phi(t)\|_{L^2} \, . $$
The first step is to show that the local existence time in Theorem \ref{Theorem1.1} in fact only depends on $I(0)$.

Even in the temporal gauge $A_0=0$ there is some freedom left for the choice of the gauge. We apply a gauge transformation with
$$ \chi(x) = (-\Delta)^{-1} div \, A(0,x) = (-\Delta)^{-1} div \,a(x) \, . $$
In the new gauge we obtain $A_0' = A_0 + \partial_t \chi = A_0 = 0$ and
$$ (A')^{cf}(0)= A^{cf}(0) + (\nabla \chi)^{cf}(0) = -(-\Delta)^{-1} \nabla div \, A(0) + (-\Delta)^{-1} \nabla div \, A(0) = 0 \, . $$
By this transformation the regularity of the data and of a solution is preserved, as we now show. The same holds for its inverse obtained by replacing $\chi$ by $-\chi$. We namely have $\|\phi'\|_{L^2} = \|\phi\|_{L^2}$ and $\|\partial_t \phi'\|_{L^2} = \|\partial_t \phi \|_{L^2}$ as well as
\begin{align*}
\|\partial_j \phi'\|_{L^2} & = \|\partial_j(e^{i \chi} \phi)\|_{L^2} \le \|(\partial_j \chi)e^{i\chi} \phi\|_{L^2} + \|e^{i\chi} \partial_j \phi\|_{L^2} \\
& \lesssim \|A(0)\|_{L^{2+\epsilon}} \|\phi\|_{L^{\frac{4+2\epsilon}{\epsilon}}} + \|\phi\|_{H^1} \lesssim (\||\nabla|^{\epsilon} a\|_{L^2} +1) \|\phi\|_{H^1} < \infty \, ,\\
\||\nabla|^{\epsilon} A'\|_{H^{\frac{1}{2}}} & \lesssim \||\nabla|^{\epsilon} A\|_{H^{\frac{1}{2}}} + \||\nabla|^{\epsilon} \nabla \chi \|_{H^{\frac{1}{2}}} \lesssim \||\nabla|^{\epsilon} A\|_{H^{\frac{1}{2}}} +\||\nabla|^{\epsilon} a\|_{H^{\frac{1}{2}}} < \infty \, .\\
\end{align*}
Moreover the compatability condition is obviously preserved. An elementary computation also shows that $I(t)$ as well as $E(t)$ is preserved, because
\begin{align*}
&\|D_{\mu}' \phi'\|_{L^2} = \|(\partial_{\mu} - i(A_{\mu} + \partial_{\mu} \chi))(e^{i\chi} \phi)\|_{L^2} \\
&=\|ie^{i\chi} \partial_{\mu} \chi \phi + e^{i\chi} \partial_{\mu} \phi - iA_{\mu} e^{i\chi} \phi - i \partial_{\mu} \chi e^{i\chi} \phi \|_{L^2} 
 = \| D_{\mu} \phi\|_{L^2} \, .
\end{align*}

We apply Theorem \ref{Theorem1.1} to the transformed problem and obtain a solution on $[0,T]$, where $T$ depends only on $\|\phi'(0)\|_{H^1} + \|\partial_t \phi'(0)\|_{L^2}$, where we used that $(A')^{cf}(0)=0$. We now show that this quantity is controlled by $I(0)$. Trivially we have
$\|\phi'(0)\|_{L^2} = \|\phi(0)\|_{L^2}$ and $\|(\partial_t \phi')(0)\|_{L^2} = \|(\partial_t \phi)(0)\|_{L^2}$. Furthermore using $(A')^{cf}(0) =0$ we obtain
\begin{equation}
\label{51}
 \|\partial_j \phi'(0)\|_{L^2} \le \|D_j' \phi'(0)\|_{L^2} + \|A_j'(0) \phi'(0)\|_{L^2} = \|D_j' \phi'(0)\|_{L^2} + \|(A_j')^{df}(0) \phi'(0)\|_{L^2} \, . \end{equation}
But now we obtain by (\ref{***2})
\begin{align}
\label{52}
(A_1')^{df}(0) &= -2 \Delta^{-1} \partial_2 Im(\overline{\phi'(0)} (\partial_t \phi'(0)) \\
\nonumber
&= -2 \Delta^{-1} \partial_2 Im(e^{-i\chi} \overline{\phi(0)} e^{i\chi} (\partial_t \phi(0)) = A_1^{df}(0)
\end{align}
and similarly $(A_2')^{df}(0) = A_2^{df}(0)$. By the covariant Sobolev inequality (cf. \cite{GV})
$$ \|\phi'(0)\|_{L^4} \lesssim \|\phi'(0)\|_{L^2}^{\frac{1}{2}} (\sum_{j=1}^2 \|D_j' \phi'(0)\|_{L^2})^{\frac{1}{2}} = \|\phi(0)\|_{L^2}^{\frac{1}{2}} (\sum_{j=1}^2 \|D_j \phi(0)\|_{L^2})^{\frac{1}{2}} \lesssim I(0) $$ and similarly $\|\phi(0)\|_{L^4} \lesssim I(0)$,
and thus by (\ref{51}) and (\ref{52})
\begin{align*}
\|\partial_j \phi'(0)\|_{L^2} & \lesssim \|D_j \phi(0)\|_{L^2} + \|A^{df}(0)\|_{L^4} \|\phi'(0)\|_{L^4} \\
& \lesssim I(0) + \| \nabla^{-1}(\phi(0) (\partial_t \phi)(0))\|_{L^4} I(0) \\
& \lesssim I(0) (1+ \|\phi(0)\|_{L^4} \|\partial_t \phi(0)\|_{L^2} ) \\
& \lesssim I(0) (1+\|I(0)\|_{L^2}^2) \, .
\end{align*}
We conclude that $T$ only depends on $I(0)$. Finally we reverse the gauge transform to obtain the solution $(\phi(t),A(t))$ on $[0,T]$.

What we need to obtain a global solution is an a priori bound of $I(t)$ on every finite time interval. Of course we use energy conservation $E(t)=E(0)$. Under our sign assumption $V(r) \ge - \alpha^2 r$ $\forall\, r \ge 0$ we obtain
\begin{equation}
\label{2.11}
\sum_{\mu=0}^2 \|D_{\mu} \phi(t)\|_{L^2}^2 = E(t) - \int V(|\phi|^2) dx \le |E(0)| + \alpha^2 \|\phi(t)\|_{L^2}^2 \, . 
\end{equation}
This implies
\begin{align*}
\frac{d}{dt}\big(\|\phi(t)\|_{L^2}^2\big) & = \int 2 Re(\overline{\phi(t)} (D_0 \phi)(t)) dx \\
& \le 2 \|\phi(t)\|_{L^2} \|(D_0 \phi)(t)\|_{L^2}  \\
& \le 2 \|\phi(t)\|_{L^2} (|E(0)| + \alpha^2 \|\phi(t)\|_{L^2}^2)^{\frac{1}{2}} \\
& \le \alpha^{-1} |E(0)| + 2 \alpha \|\phi(t)\|_{L^2}^2 \, ,
\end{align*}
hence by Gronwall's lemma
\begin{equation}
\label{2.12}
\| \phi(t) \|_{L^2}^2 \le e^{2\alpha |t|} (\|\phi(0)\|_{L^2}^2 + |t| \alpha^{-1} |E(0)|) \, .
\end{equation}
By (\ref{2.11}) and (\ref{2.12}) we obtain the desired a priori control of $I(t)$, so that Theorem \ref{Theorem1.2} is proved.
\end{proof}


\begin{thebibliography}{999999}
\bibitem[AFS]{AFS} P. d'Ancona, D. Foschi, and S. Selberg: {\sl Product estimates for wave-Sobolev spaces in 2+1 and 1+1 dimensions}. 
Contemporary Math. 526 (2010), 125-150
\bibitem[B]{B} N. Bournaveas: {\sl Low regularity solutions of the Chern-Simons-Higgs equations in the Lorentz gauge}. Electr. J. Diff. Equ. 2009 (2009) No. 114, 1-10
\bibitem[CC]{CC} D. Chae and K.Choe: {\sl Global existence in the Cauchy problem of the relativistic Chern-Simons-Higgs theory}. Nonlinearity 15 (2002), 747-758
\bibitem[HKP]{HKP} J. Hong, Y. Kim and P.Y. Pac: {\sl Multivortex solutions of the abelian Chern-Simons-Higgs theory}. Phys. Rev. Letters 64 (1990), 2230-2233
\bibitem[GV]{GV} J. Ginibre and G. Velo: {\sl The Cauchy problem for coupled Yang-Mills and scalar fields in the temporal gauge}. Comm. Math. Phys. 82 (1981), 1-28
\bibitem[H]{H} H. Huh: {\sl Local and global solutions of the Chern-Simons-Higgs system}. J. Funct. Anal. 242 (2007), 526-549
\bibitem[HO]{HO} H. Huh and S.-J. Oh: {\sl Low regularity solutions to the Chern-Simons-Dirac and the Chern-Simons-Higgs equations in the Lorenz gauge}. Preprint arXiv:1209.3841 
\bibitem[O]{O} S.-J. Oh: {\sl Finite energy global well-posedness of the Chern-Simons-Higgs equations in the Coulomb gauge}. Preprint arXiv:1310.3955
\bibitem[JW]{JW} R. Jackiw and E.J. Weinberg: {\sl Self-dual Chern-Simons vortices}. Phys. Rev. Letters 64 (1990), 2234-2237
\bibitem[KMBT]{KMBT} S. Klainerman and M. Machedon (Appendices by J. Bougain and  D. Tataru): {\sl Remark on Strichartz-type inequalities}. Int. Math. Res. Notices 1996, no.5, 201-220
\bibitem[S]{S} S. Selberg: {\sl Anisotropic bilinear $L^2$ estimates related to the 3D wave equation}. Int. Math. Res. Not. 2008 (2008), No. 107
\bibitem[SO]{SO} S. Selberg and D. Oliveira da Silva: {\sl A remark on unconditional uniqueness in the Chern-Simons-Higgs model}. Diff. Int. Equs. 28 (2015), 333 - 346
\bibitem[ST]{ST} S.Selberg and A. Tesfahun: {\sl Global well-posedness of the Chern-Simons-Higgs equations with finite energy}. Discrete Cont. Dyn. Syst. 33 (2013), 2531-254
\bibitem[Str]{Str} R.S. Strichartz: {\sl Restrictions of Fourier transforms to quadratic surfaces and decay of solutions of wave equations}. Duke Math. J. 44 (1977), 705-714
\bibitem[T]{T} T. Tao: {\sl Multilinear weighted convolution of $L^2$-functions and applications to non-linear dispersive equations}. Amer. J. Math. 123 (2001), 838-908
\bibitem[T1]{T1} T. Tao: {\sl Local well-posedness of the Yang-Mills equation in the temporal gauge below the energy norm}. J. Diff. Equ. 189 (2003), 366-382
\bibitem[Y]{Y} J. Yuan: {\sl Local well-posedness of Chern-Simons-Higgs system in the Lorentz gauge}. J. Math. Phys. 52 (2011), 103706
\end{thebibliography}
\end{document}